\documentclass[10pt,leqno]{article} 

   \usepackage[centertags]{amsmath}
   \usepackage{amsfonts}
   \usepackage{amsmath}
   \usepackage{amssymb}
   \usepackage{amsthm}
   \usepackage{newlfont}
   \usepackage[latin1]{inputenc}

 \usepackage{multirow, bigdelim}

\usepackage{color}

\allowdisplaybreaks[3]

\theoremstyle{plain}
\newtheorem{theorem}{Theorem}[section]
\newtheorem{lemma}[theorem]{Lemma}

\newtheorem{corollary}[theorem]{Corollary}

\theoremstyle{definition}
\newtheorem{definition}[theorem]{Definition}

\theoremstyle{remark}
\newtheorem{remark}[theorem]{Remark}
\newtheorem*{remark*}{Remark}

\numberwithin{equation}{section}

\newcommand{\dosfilas}[2]{
  \ldelim[{2}{2mm}& #1 &\rdelim]{2}{2mm} \\
  & #2 & &  & &
}

\catcode`,\active

\catcode`\,12

\newcommand\D{{\mathcal D}}

\newcommand\F{{\mathcal F}}
\newcommand\G{{\mathcal G}}
\newcommand\U{{\mathcal U}}
\newcommand\Hh{{\mathcal H}}

\newcommand\CC{{\mathbb C}}

\newcommand\RR{{\mathbb R}}
\newcommand\ZZ{{\mathbb Z}}
\newcommand\NN{{\mathbb N}}

\newcommand\x{{\theta}}

\newcommand\Sh{\mbox{\Large $\mathfrak {s}$}}

   \title{New examples of Krall-Meixner and Krall-Hahn polynomials, with applications to the construction of exceptional Meixner and Laguerre polynomials.
   \footnote{Partially supported by PGC2018-096504-B-C31
(FEDER(EU)/Ministerio de Ciencia e Innovaci\'on-Agencia Estatal de Investigaci\'on),
FQM-262 and Feder-US-1254600 (FEDER(EU)/Jun\-ta de Anda\-lu\-c\'ia).}}
   \author{Antonio J. Dur\'{a}n\\
     \footnotesize
        \  Departamento de An\'{a}lisis Matem\'{a}tico.
       Universidad de Sevilla \\
       \footnotesize Apdo (P. O. BOX) 1160. 41080 Sevilla. Spain.
   duran@us.es \\
          \ \ }
   \date{}
   
\begin{document}
   \maketitle

\bigskip

\begin{abstract}
We construct new examples of Krall discrete orthogonal polynomials, i.e., orthogonal polynomials with respect to a measure which are also eigenfunctions of a higher order difference operator.  The new examples include the orthogonal polynomials with respect to the measures obtained from the Meixner measure $\rho _{a,c}$ and Hahn measure $\rho _{a,b,N}$ by removing a finite number of their mass points when the parameter $c$ of the Meixner measure and $a$ or $b$ of the Hahn measure are positive integers. From the new Krall-Meixner families we construct new families of exceptional Meixner and Laguerre polynomials.
\end{abstract}

\section{Introduction}
The most important families of orthogonal polynomials are the classical and classical discrete. Besides the orthogonality, they are also common eigenfunctions of a second order differential or difference operator, respectively (the $q$-classical families are not considered in this paper).

Eighty years ago H.L. Krall raised the issue of orthogonal polynomials which are also common eigenfunctions of a higher order differential operator. He obtained a complete classification for the case of a differential operator of order four (\cite{Kr2}). After his pioneer work, orthogonal polynomials which are also common eigenfunctions of higher order differential operators are usually called Krall polynomials.  Since the eighties a lot of effort has been devoted to find Krall polynomials (\cite{koekoe}, \cite{koe} \cite{koekoe2}, \cite{L1}, \cite{L2}, \cite{GrH1}, \cite{GrHH}, \cite{GrY}, \cite{Plamen1}, \cite{Plamen2}, \cite{Zh}, the list is by no mean exhaustive).

Discrete versions of the Krall problem appeared at the beginning of the nineties. Richard Askey explicitly posed in 1991 (see page 418 of \cite{BGR}) the problem of finding orthogonal polynomials which are also common eigenfunctions of a higher order difference operator (discrete Krall polynomials) of the form
\begin{equation}\label{hodo}
\sum_{l=s}^rh_l\Sh _l, \quad s\le r, s,r\in \ZZ,
\end{equation}
where $h_l$ are polynomials and $\Sh_l$ stands for the shift operator $\Sh_l(p)=p(x+l)$.

But the first examples
of discrete Krall polynomials needed more than twenty years to be constructed: a huge amount of families of Krall discrete orthogonal polynomials were introduced by the author by mean of a certain Christoffel transform of the classical discrete measures of Charlier, Meixner, Krawtchouk and Hahn (see \cite{du0}, \cite{du1}, \cite{dudh}, \cite{DdI}, \cite{DdI2}). A Christoffel transform is a transformation which consists in multiplying a measure $\mu$ by a polynomial $r$. For the Meixner case, the Krall-Meixner measures introduced by the author are the following:
\begin{equation}\label{roi}
\rho_{a,c}^\F=\sum _{x=0}^\infty \prod _{f\in F_1}(x-f)\prod _{f\in F_2}(x+c+f)\frac{a^x\Gamma(x+c)}{x!}\delta_x,
\end{equation}
with $\F=(F_1,F_2)$, where $F_i$, $i=1,2$, are finite sets of positive integers, $a\not =0,1$ and $c\not =0,-1,-2,\cdots$.

In the terminology introduced by Duistermaat and Gr\"unbaum \cite{DG} (see also \cite{GrH1}, \cite{GrH3})),  Krall and Krall discrete polynomials are examples of the so-called bispectral polynomials, because with these families $(q_n(x))_n$ of polynomials are associated two  operators with respect to which they are eigenfunctions: one acting in the discrete variable $n$ (the three term recurrence relation associated to the orthogonality with respect to a measure in the real line) and the other in the continuous variable $x$.

The purpose of this paper is to introduce some new examples of Krall-Meixner and Krall-Hahn measures. In particular, these new families include the measures obtained from the Meixner measures $\rho _{a,c}$ and Hahn measures $\rho _{a,b,N}$ by removing a finite number of their mass points when the parameter $c$ of the Meixner measure and $a$ or $b$ of the Hahn measure are positive integers. These measures can not be obtained by any Christoffel transform applied to a Meixner or Hahn measures.

We consider the Meixner case in the Section \ref{sect3}. Surprisingly enough, the new Krall-Meixner measures appear by taking limits in the Christoffel transforms (\ref{roi}) when $c$ goes to $\hat c\in \{0,-1,-2,\cdots\}$, under the additional assumption
$$
\{0,1,\cdots , -\hat c\}\subset F_1\cup (-\hat c-F_2).
$$
We then get measures of the form
\begin{equation}\label{nui}
\sum _{x\in \NN\setminus F_1} \prod _{h\in \Hh}(x-h)a^x\delta_x,
\end{equation}
where $\Hh=[F_1\cup (-\hat c-F_2)\setminus \{0,1,\cdots , -\hat c\}]\cup [F_1\cap (-\hat c-F_2)]$. The measure (\ref{nui}) is positive when $\prod _{h\in \Hh}(x-h)\ge 0$, $x\in \NN\setminus F_1$.
We also construct explicitly the ortogonal polynomials with respect to the measure (\ref{nui}) (see Theorem \ref{1th}).

The situation is similar in the Hahn case which we consider in the Section \ref{sect4}.

These new examples of Krall discrete orthogonal polynomials are also interesting by the following reason. As it has been shown in \cite{duch}, \cite{dume} and \cite{duha}, exceptional discrete polynomials can be constructed by applying duality (in the sense of \cite{Leo}) to Krall-Charlier, Krall-Meixner or Krall-dual Hahn orthogonal polynomials. Passing then to the limit, exceptional Hermite, Laguerre and Jacobi polynomials can be constructed.
Exceptional and exceptional discrete orthogonal polynomials $p_n$, $n\in X\varsubsetneq \NN$, are complete orthogonal polynomial systems with respect to a positive measure which in addition are eigenfunctions of a second order differential or difference operator, respectively. They extend the  classical families of Hermite, Laguerre and Jacobi or the classical discrete families of Charlier, Meixner and Hahn.

The last thirteen years have seen a great deal of activity in the area  of exceptional orthogonal polynomials (see, for instance,
\cite{Be,BK,duch,dume,duha,GFGM,GUKM1,GUKM2} (where the adjective \textrm{exceptional} for this topic was introduced),  \cite{GUGM,GQ,MR,OS3,OS4,STZ}, and the references therein).
The most apparent difference between classical or classical discrete orthogonal polynomials and their exceptional counterparts
is that the exceptional families have gaps in their degrees, in the
sense that not all degrees are present in the sequence of polynomials (as it happens with the classical families) although they form a complete orthonormal set of the underlying $L^2$ space defined by the orthogonalizing positive measure. This
means in particular that they are not covered by the hypotheses of Bochner's and Lancaster's classification theorems for classical and classical discrete orthogonal polynomials, respectively (see \cite{B} or \cite{La}).

We complete this paper by constructing new examples of exceptional Meixner and exceptional Laguerre polynomials using the new examples of Krall-Meixner polynomials (see Sections \ref{sect5} and \ref{sect6}, respectively). The examples already known of exceptional Laguerre polynomials are orthogonal with respect to weights of the form
\begin{equation}\label{lai}
\frac{x^\alpha e^{-x} dx}{\tau^2(x)},\quad x\in (0,+\infty ),
\end{equation}
where $\alpha\not =0$ and $\tau$ is a polynomial which does not vanish in $[0,+\infty)$. In the Section \ref{sect6} of this paper we construct the first exceptional Laguerre polynomials associated to a weight of the form (\ref{lai}) with $\alpha=0$.

In a subsequent paper, we will study de dual-Hahn case which will generate new exceptional Hahn and Jacobi polynomials. The dual-Hahn case is even much more interesting than the Meixner or Hahn cases because, besides the discrete parameters collected by the elements of the finite sets of positive integers $F_1$ and $F_2$, the new families also contain an arbitrary number of continuous deformation parameters (see \cite{GGMx}, where the first examples of exceptional Legendre polynomials depending of an arbitrary number of continuous deformation parameters have been introduced). But  this reason makes the dual-Hahn case much more difficult, because it can not be managed passing to the limit as the Meixner or Hahn cases.

\section{Preliminaries}

Let $\mu$ be a real function defined in a set $X$ of real numbers which it is either an interval (bounded or not) or countable. We also denote by  $\mu$ the following measure: (a) if $X$ is an interval then $\mu$ is the measure with density $\mu(x)$ with respect to the Lebesgue measure; if $X$ is countable then $\mu$ is the discrete measure with masses $\mu(x)$ at $x$: $\mu=\sum_{x\in X}\mu(x) \delta_x$.

\begin{lemma}\label{lemi} Let $\mu_s$, $s\in \NN$, and $\mu$ be real functions defined in a set $X$ of real numbers which it is either an interval (bounded or not) or countable. Assume that
\begin{enumerate}
\item for $n\in \NN$, $\vert x^n\mu_s(x)\vert \le f_n(x)$, $x\in X$, and $f_n\in L^1(X)$.
\item For $n\in \NN$, $x^n\mu(x)\in L^1(X)$.
\item For $x\in X$, $\lim \mu_s(x)=\mu(x)$.
\end{enumerate}
Assume that $(p_n^s)_n$, $(p_n)_n$ are the sequences of monic orthogonal polynomials with respect to $\mu_s$ and $\mu$, respectively. Then for $n\in \NN$,
$\lim_{s\to \infty} p_n^s(z)=p_n(z)$, $z\in \CC$, $\lim_{s\to \infty} \langle p_n^s,p_n^s\rangle _{\mu_s}=\langle p_n,p_n\rangle_\mu$. Moreover if the measures $\mu_s$ and $\mu$ are positive and $(p_n^s)_n$ are complete in $L^2(\mu_s)$
then $(p_n)_n$ are also complete in $L^2(\mu)$.
\end{lemma}

\begin{proof}
It is an easy consequence of the Lebesgue's dominated convergence Theorem.
\end{proof}

For a discrete measure $\rho=\sum_{x\in \NN} a_x\delta _x$ and $u\in \NN$, we denote by $\rho(x-u)$ the translate measure
\begin{equation}\label{mtr}
\rho(x-u)=\sum_{x=u}^\infty a_{x-u}\delta_x.
\end{equation}

\bigskip
We finish the Preliminaries will a couple of definitions.

Given a finite set of numbers $F=\{f_1,\cdots, f_k\}$, $f_i<f_j$ if $i<j$, we denote by $V_F$ the Vandermonde determinant defined by
\begin{equation}\label{defvdm}
V_F=\prod_{1=i<j=k}(f_j-f_i).
\end{equation}
For an nonnegative integer $u$ and a couple of finite sets $I,J$ of positive integers, we define the finite set of integers $\Hh (u,I,J)$ as
\begin{equation}\label{spmh}
\Hh (u,I,J)=[[I\cup (u-J)]\setminus \{0,1,2,\cdots ,u\}]\cup [I\cap (u-J)].
\end{equation}

\subsection{Finite sets and pair of finite sets of positive integers.}\label{sfspi}
Let $F$ be a finite set of positive integers (we always arrange the elements of $F$ in increasing order).

Consider the set $\Upsilon$  formed by all finite sets of positive
integers:
\begin{equation*}
\Upsilon=\{F:\mbox{$F$ is a finite set of positive integers}\} .
\end{equation*}
We consider the involution $I$ in $\Upsilon$ defined by
\begin{align}\label{dinv}
I(F)=\{1,2,\cdots, \max F\}\setminus \{\max F-f,f\in F\}.
\end{align}
For $F=\emptyset$, we define $\max F=\min F=-1$, and so $I(\emptyset)=\emptyset$.

The definition of $I$ implies that $I^2=Id$.

The set $I(F)$ will be denoted by $G$: $G=I(F)$. Notice that
$$
\max F=\max G,\quad m=\max F-k+1,
$$
where $k$ and $m$ are the number of elements of $F$ and $G$,
respectively.

For a finite set $F=\{f_1,\cdots ,f_{n_F}\}$, $f_i<f_{i+1}$, of
positive integers, we define the number $s_F$ by
\begin{equation}\label{defs0}
s_F=\begin{cases} 1,& \mbox{if $F=\emptyset$},\\
n_F+1,&\mbox{if $F=\{1,2,\cdots , n_F\}$},\\
\min \{s\ge 1:s<f_s\}, & \mbox{if $F\not =\{1,2,\cdots n_F\}$}.
\end{cases}
\end{equation}
If $F$ is a finite set of nonnegative integers with $0\in F$, we set $s_F=s_{F\setminus \{0\}}$, and define the set $F_{\Downarrow}$ of positive integers by
\begin{equation}\label{defffd}
F_{\Downarrow}=\begin{cases} \emptyset,& \mbox{if $F=\emptyset$ or $F=\{1,2,\cdots , n_F\}$,}\\
\{f_{s_F}-s_F,\cdots , f_{n_F}-s_F\},& \mbox{if $F\not =\{1,2,\cdots ,
k\}$ and $0\not \in F$},\\
(F\setminus \{0\})_{\Downarrow},& \mbox{if $0 \in F$}.
\end{cases}
\end{equation}
Notice that if $F\not =\emptyset$ then
\begin{equation}\label{lum}
F=\begin{cases}\{1,\cdots , s_F-1\}\cup (s_F+F_{\Downarrow}), &\mbox{if $0\not \in F$ and $s_F>1$}\\
\{0,1,\cdots , s_F-1\}\cup (s_F+F_{\Downarrow}), &\mbox{if $0\in F$.}\end{cases}
\end{equation}

From now on, $\F=(F_1,F_2)$ will denote a pair of finite sets of
positive integers. We denote by $k_j$ the number of elements of $F_j$,
$j=1,2$, and $k=k_1+k_2$ is the number of elements of $\F$. One of
the components of $\F$, but not both, can be the empty set.

We associate to $\F$ the nonnegative integer $u_\F$  and the infinite set of nonnegative integers $\sigma_\F$ defined by
\begin{align}\label{defuf}
u_\F&=\sum_{f\in F_1}f+\sum_{f\in
F_2}f-\binom{k_1+1}{2}-\binom{k_2}{2},\\\label{defsf}
\sigma _\F&=\{u_\F,u_\F+1,u_\F+2,\cdots \}\setminus \{u_\F+f,f\in
F_1\}.
\end{align}
The infinite set $\sigma_\F$ will be the set of indices for the exceptional Meixner or Laguerre polynomials associated to $\F$.

For a pair $\F=(F_1,F_2)$ we denote by
$\F_{\Downarrow}$ the pair of finite sets of positive integers
defined by
\begin{equation}\label{deffd}
\F_{\Downarrow}=((F_1)_{\Downarrow},F_2),
\end{equation}
where $(F_1)_{\Downarrow}$ is defined by (\ref{defffd}). We also define
\begin{equation}\label{defs0f}
s_\F=s_{F_1}
\end{equation}
where the number $s_{F_1}$ is defined by (\ref{defs0}).

\subsection{Meixner and Laguerre polynomials}
We include here basic definitions and facts about Meixner and Laguerre polynomials, which we will need in the following Sections.

For $a\not =0, 1$ we write $(m_{n}^{a,c})_n$ for the sequence of Meixner polynomials defined by
\begin{equation}\label{Mxpol}
m_{n}^{a,c}(x)=\frac{a^n}{(1-a)^n}\sum _{j=0}^n a^{-j}\binom{x}{j}\binom{-x-c}{n-j}
\end{equation}
(we have taken a slightly different normalization from the one used in \cite{KLS}, pp, 234-7).

For $0<\vert a\vert<1$ and $c\not =0,-1,-2,\ldots $, they are orthogonal with respect to the measure
\begin{equation*}\label{MXw}
\rho_{a,c}=\sum _{x=0}^\infty \frac{a^{x}\Gamma(x+c)}{x!}\delta _x,
\end{equation*}
and
\begin{equation}\label{norme}
\langle m_n^{a,c},m_n^{a,c}\rangle =\frac{a^{n}\Gamma(n+c)}{n!(1-a)^{2n+c}}.
\end{equation}
When $c\in \{0,-1,-2,\ldots \}$, the Meixner polynomials are not orthogonal. In this case, the sequence of Meixner polynomials is actually \textsl{packing} two sequences of orthogonal polynomials. Indeed, on the one hand, the finite sequence of polynomial $m_n^{a,c}$, $n=0,\cdots ,-c$, is the sequence of Krawtchouk polynomials. On the other hand, for $n\ge -c+1$, we have
$$
m_n^{a,c}(x)=\frac{\prod_{j=0}^{-c}(x-j)}{\prod_{j=0}^{-c}(n-j)}m_{n+c-1}^{a,2-c}(x+c-1),
$$
and since $2-c$ is a positive integer, $(m_{n}^{a,2-c}(x))_n$, is a sequence of Meixner orthogonal polynomials.

For $\alpha\in\mathbb{R}$, we write $(L_n^\alpha )_n$ for the sequence of Laguerre polynomials
\begin{equation}\label{deflap}
L_n^{\alpha}(x)=\sum_{j=0}^n\frac{(-x)^j}{j!}\binom{n+\alpha}{n-j}
\end{equation}
(see  \cite{KLS}, pp, 241-244).

For $\alpha\neq-1,-2,\ldots$, they are orthogonal with respect to a measure $\mu_{\alpha}=\mu_{\alpha}(x)dx$. This measure is positive
only when $\alpha>-1$ and then
$$
\mu_{\alpha}(x) =x^\alpha e^{-x}, x>0.
$$
One can obtain Laguerre polynomials from Meixner polynomials using the limit
\begin{equation}\label{blmel}
\lim_{a\to 1}(a-1)^nm_n^{a,c}\left(\frac{x}{1-a}\right)=L_n^{c-1}(x)
\end{equation}
see \cite{KLS}, p. 243 (take into account that we are using for the
Meixner polynomials a different normalization to that in \cite{KLS}). The previous limit is uniform in compact sets of $\CC$.

\section{New Meixner-Krall orthogonal polynomials}\label{sect3}
In \cite{du0}, we conjecture how to construct Krall discrete orthogonal polynomials by using Christoffel transforms of the classical discrete measure. A method (using the new concept of $\D$-operators) was developed in \cite{du1} to prove those conjectures. The method was refined in \cite{DdI,DdI2}, where we completed the proof of all the conjectures in \cite{du0}.

In the Meixner case, we consider the Christoffel transform $\rho_{a,d}^\F$ of the Meixner measure defined as follows: for $0<\vert a\vert <1$, $d\not \in \{0,-1,-2,\cdots \}$ and a pair $\F=(F_1,F_2)$ of finite sets $F_i$, $i=1,2$, of positive integers,
\begin{equation}\label{mqs}
\rho_{a,d}^\F=\sum _{x=0}^\infty \prod _{f\in F_1}(x-f)\prod _{f\in F_2}(x+d+f)\frac{a^x\Gamma(x+d)}{x!}\delta_x.
\end{equation}
It was proved in \cite{DdI} that the orthogonal polynomials with respect to the measure $\rho_{a,d}^\F$ (when there exist) are also eigenfunctions of a higher order difference operator of the form (\ref{hodo}).

The most interesting case is when $\rho_{a,d}^\F$ is a positive measure, i.e., when $d$ and $\F$ satisfy
$$
\prod _{f\in F_1}(x-f)\prod _{f\in F_2}(x+d+f)\Gamma(x+d)\ge 0, \quad x\in \NN.
$$
The case $d\in \{0,-1,-2,\cdots \}$ was neither considered in \cite{du0} nor in the subsequence papers because the Meixner measure
$$
\sum _{x=0}^\infty \frac{a^x\Gamma(x+d)}{x!}\delta_x
$$
is not well defined since the Gamma function has poles at the non-positive integers.

Surprisingly enough, for certain pairs $\F$, the measure (\ref{mqs}) makes sense even when $d\in \{0,-1,-2,\cdots \}$. Indeed,
for $d=\hat c\in \{0,-1,-2,\cdots \}$,  and a pair $\F=(F_1,F_2)$ of finite sets $F_i$, $i=1,2$, of positive integers satisfying
\begin{equation}\label{hf2}
\{0,1,2,\cdots , -\hat c\}\subset F_1\cup (-\hat c-F_2),
\end{equation}
define the finite set $\Hh$ of integers
\begin{equation}\label{defH}
\Hh=[F_1\cup (-\hat c-F_2)\setminus \{0,1,2,\cdots , -\hat c\}]\cup [F_1\cap (-\hat c-F_2)].
\end{equation}
Actually $\Hh$ depends on $\hat c$, $F_1$ and $F_2$ (in fact, it is the finite set we denote by $\Hh(-\hat c,F_1,F_2)$ in the preliminaries (\ref{spmh})), although to simplify the notation we will not write explicitly this dependency.

Consider now the measure
\begin{equation}\label{defnu}
\nu ^{a}_{\hat c;\F}=\sum _{x=0;x\not \in F_1}^\infty \prod _{h\in \Hh}(x-h)a^x\delta_x.
\end{equation}

We next prove that the measure $\nu ^{a}_{\hat c;\F}$ is actually a limit of measures of the form (\ref{mqs}) when $d\to \hat c$. We will see later on that the orthogonal polynomials with respect to $\nu ^{a}_{\hat c;\F}$ (when there exist) are also eigenfunctions of a higher order difference operator of the form (\ref{hodo}).

\begin{lemma}\label{1le} Let $\hat c\in \{0,-1,-2,\cdots \}$,  and let $\F=(F_1,F_2)$ be a pair of finite sets $F_i$, $i=1,2$, of positive integers satisfying (\ref{hf2}).
For a sequence of real numbers $\hat c_s\not\in \{0,-1,-2,\cdots \}$ with $\lim_{s\to \infty}\hat c_s=\hat c$ we have
\begin{equation}\label{limit}
\lim_{s\to \infty}\rho_{a,\hat c_s}^\F(x)=\nu ^{a}_{\hat c;\F}(x),\quad x\in \NN,
\end{equation}
where the measures $\rho_{a,\hat c_s}^\F$ and $\nu ^{a}_{\hat c;\F}$ are defined by (\ref{mqs}) and (\ref{defnu}), respectively.
Moreover, if the measure $\nu ^{a}_{\hat c;\F}$ is positive then the numbers $\hat c_s\not\in \{0,-1,-2,\cdots \}$ can be taken so that the measures
$\rho_{a,\hat c_s}^\F$ are positive as well.
\end{lemma}

\begin{proof}
The limit is straightforward for $y\in F_1$, because then $\rho_{a,\hat c_s}^\F(y)=\nu_{a,\hat c}^\F(y)=0$.

For a nonnegative integer $y\ge -\hat c+1$, using (\ref{hf2}) and (\ref{defH}), we have for the mass at $x=y$
\begin{align*}
\lim _{s\to +\infty}\rho_{a,\hat c_s}^\F(y)
&=\lim _{s\to +\infty}\prod _{f\in F_1}(y-f)\prod _{f\in F_2}(y+\hat c_s+f)\frac{a^y\Gamma(y+\hat c_s)}{y!}\\
&=\prod _{f\in F_1}(y-f)\prod _{f\in F_2}(y+\hat c+f)\frac{a^y\Gamma(y+\hat c)}{y!}\\
&=\prod _{h\in \Hh}(y-h)a^y\frac{y(y-1)\cdots (y+\hat c)\Gamma(y+\tilde c)}{y!}\\
&=\prod _{h\in \Hh}(y-h)a^y=\nu_{a,\hat c}^\F(y).
\end{align*}
Assume now that the nonnegative integer $y$ satisfies $0\le y\le -\hat c$
and $y\not \in F_1$. Then $y\in -\hat c-F_2$, and hence there is $f_y\in F_2$ with
$y+\hat c+f_y=0$. This gives
$$
\lim _{s\to +\infty }(y+\hat c_s+f_y)\Gamma(y+\hat c_s)=\frac{(-1)^{y+\hat c}}{(-y-\hat c)!}.
$$
And then
\begin{align*}
\lim _{s\to +\infty}\rho_{a,\hat c_s}^\F(y)&=\lim _{s\to +\infty }\prod _{f\in F_1}(y-f)\prod _{f\in F_2}(y+\hat c_s+f)\frac{a^y\Gamma(y+\hat c_s)}{y!}\\
&=\lim _{s\to +\infty }\prod _{f\in F_1}(y-f)\prod _{f\in F_2;f\not =f_y}(y+\hat c_s+f)\frac{a^y}{y!}(y+\hat c_s+f_y)\Gamma(y+\hat c_s)\\
&=\prod _{h\in \Hh}(y-h)\frac{a^y}{y!}y(y-1)\cdots 1 \cdot 1\cdot 2\cdots (y+\hat c)\frac{(-1)^{y+\hat c}}{(-y-\hat c)!}\\
&=\prod _{h\in \Hh}(y-h)a^y=\nu_{a,\hat c}^\F(y).
\end{align*}

Assume now that $\nu ^{a}_{\hat c;\F}$ is positive. If we write $M=\max (F_1\cup (-\hat c-F_2)\cup \{-\hat c+1\})$, then for each number $\hat c_s\not\in \{0,-1,-2,\cdots \}$ with $\hat c_s>\hat c-1$, it is easy to check that the masses $\rho_{a,\hat c_s}^\F(x)$ are always positive for $x\ge M$. Since for $x\in \{0,1,\cdots, M\}\setminus (\Hh\cup F_1)$, $\nu ^{a}_{\hat c;\F}(x)>0$ and $\{0,1,\cdots, M\}\setminus (\Hh\cup F_1)$ is finite, taking into account that we have already proved the limit (\ref{limit}), we can take  numbers $\hat c_s\not\in \{0,-1,-2,\cdots \}$ with $\hat c_s>\hat c-1$, $\lim_{s\to +\infty}\hat c_s=\hat c$ and such that $\rho_{a,\hat c_s}^\F(x)>0$, $x\in \{0,1,\cdots, M\}\setminus (\Hh\cup F_1)$. Finally, for $x\ge 0$ and $x\in \Hh\cup F_1$, we have that $x\in F_1$; indeed, if $x\not \in F_1$, then $x\in \Hh$, but $x\not \in -\hat c-F_2$, because otherwise $x\le -\hat c$ and according with the definition of $\Hh$, $x$ must then be negative. Hence $x\in F_1$  and then $\rho_{a,\hat c_s}^\F(x)=0$. Hence the measures $\rho_{a,\hat c_s}^\F$ are positive.
\end{proof}

Before going on with the construction of a sequence of orthogonal polynomials with respect to the measure $\nu ^{a}_{\hat c;\F}$, we make four remarks.

\begin{remark}\label{remp}
Firstly we show that we can assume that $\hat c\in \{-1,-2,\cdots \}$ because the measure $\nu ^{a}_{0;\F}$ can always be written in the form (\ref{mqs}) (up to a linear change of variable): \textit{there exist a pair $\U=(U_1,U_2)$ of finite sets of positive integers and  numbers $d, s, C$, with $d\in \{1,2,\cdots \}$, $s\in \NN$ and $C>0$ such that
$$
\nu ^{a}_{0;\F}=C\rho_{a,d}^\U(x-s).
$$}
Indeed, the condition (\ref{hf2}) then says that $0\in F_1\cup (-F_2)$. Assume first that $0\in F_1$. Consider the finite set of positive integers $(F_1)_\Downarrow$ (see (\ref{defffd})). We then use (\ref{lum}):
$$
F_1=\{0,1,\cdots, s_{F_1}-2,s_{F_1}-1\}\cup (s_{F_1}+(F_1)_\Downarrow).
$$
We also have $\Hh=[F_1\setminus \{0\}]\cup (-F_2)$ (\ref{defH}). It is then easy to check that the measure $\nu ^{a}_{0;\F}$ can be written in the form
\begin{align*}
\nu ^{a}_{0;\F}&=\sum _{x=s_{F_1}}^\infty (x-1)\cdots (x-s_{F_1}+1)\prod _{f\in s_{F_1}+(F_1)_\Downarrow}(x-f)\prod _{f\in F_2}(x+f)a^x\delta_x\\
&=\sum _{x=s_{F_1}}^\infty \prod _{f\in s_{F_1}+(F_1)_\Downarrow}(x-f)\prod _{f\in F_2}(x+f)\frac{a^x\Gamma (x)}{(x-s_{F_1})!}\delta_x.
\end{align*}
If $0\not \in F_2$, the identity (\ref{mtr}) shows that $\nu ^{a}_{0;\F}=a^{s_{F_1}}\rho_{a,s_{F_1}}^\U(x-s_{F_1})$, where $\U=(U_1,U_2)$, and the finite sets $U_i$, $i=1,2$, of positive integers are given by $U_1=(F_1)_\Downarrow$, $U_2=F_2$. If $0\in F_2$, since
$$
F_2=\{0,1,\cdots, s_{F_2}-2,s_{F_2}-1\}\cup (s_{F_2}+(F_2)_\Downarrow),
$$
we have
\begin{align*}
\nu ^{a}_{0;\F}&=\sum _{x=s_{F_1}}^\infty \prod _{f\in s_{F_1}+(F_1)_\Downarrow}(x-f)\prod _{f\in F_2}(x+f)\frac{a^x\Gamma (x)}{(x-s_{F_1})!}\delta_x\\
&=\sum _{x=s_{F_1}}^\infty \prod _{f\in s_{F_1}+(F_1)_\Downarrow}(x-f)\prod _{f\in s_{F_2}+(F_2)_\Downarrow}(x+f)\frac{a^x\Gamma (x+s_{F_2})}{(x-s_{F_1})!}\delta_x.
\end{align*}
And again the identity (\ref{mtr}) shows that $\nu ^{a}_{0;\F}=a^{s_{F_1}}\rho_{a,s_{F_1}+s_{F_2}}^{\U}(x-s_{F_1})$, where $\U=(U_1,U_2)$, and the finite sets $U_i$, $i=1,2$, of positive integers are given by $U_1=(F_1)_\Downarrow$, $U_2=(F_2)_\Downarrow$.

Finally, if $0\not \in F_1$ then $0\in F_2$, and we also have $\Hh=[F_1\cup (-F_2)] \setminus \{0\}$ (\ref{defH}).
If we denote $U_2=(F_2)_\Downarrow$, then we straightforwardly have (see (\ref{lum}))
\begin{align*}
\nu ^{a}_{0;\F}&=\sum _{x=0}^\infty \prod _{f\in F_1}(x-f)\prod _{f\in U_2}(x+s_{F_2}+f)\frac{a^x\Gamma (x+s_{F_2})}{x!}\delta_{x}=\rho_{a,s_{F_2}}^{(F_1,U_2)}.
\end{align*}
\end{remark}

\begin{remark}\label{rems}
Secondly we show that we can also assume that $F_1$ and $F_2$ are finite sets of positive integers: \textit{For $\hat c\in \{0,-1,-2,\cdots \}$, if we assume that $0\in F_1\cup F_2$, then either there exist a pair $\U=(U_1,U_2)$ of finite sets of positive integers and numbers  $s\in \NN$ and $\hat d\in \{-1,-2,\cdots \}$
such that
$$
\nu ^{a}_{\hat c;\F}=a^{s}\nu ^a_{\hat d,\U}(x-s),
$$
or there exist a pair $\U=(U_1,U_2)$ of finite sets of nonnegative integers and a number  $s\in \NN$
such that
$$
\nu ^{a}_{\hat c;\F}=a^{s}\nu ^a_{0,\U}(x-s).
$$
}
Assume first that $0\in F_1$ and $s_{F_1}\le -\hat c$. Proceeding as before, it is not difficult to check that
by taking $U_1=(F_1)_\Downarrow$ and $U_2=F_2$ then
$$
\nu_{\hat c,\F}^a=a^{s_{F_1}}\nu_{\hat c+s_{F_1},\U}^a(x-s_{F_1}).
$$
And $U_1$ is a finite set of positive integers.

If $0\in F_1$ and $s_{F_1}> -\hat c$, we then have
$$
\nu_{\hat c,\F}^a=a^{-\hat c}\nu_{0,\U}^a(x+\hat c),
$$
where $\U=(U_1,U_2)$, and the finite sets $U_i$, $i=1,2$, of nonnegative integers are given by
$$
U_1=\{0,1,\cdots, s_{F_1}-1+\hat c\}\cup (s_{F_1}+\hat c+(F_1)_\Downarrow),\quad  U_2=F_2.
$$
Similarly, if $0\not \in F_1$ and $0\in F_2$, we have for $s_{F_2}\le -\hat c$
$$
\nu_{\hat c,\F}^a=\nu_{\hat c+s_{F_2},\U}^a,
$$
where $\U=(U_1,U_2)$, and the finite sets $U_i$, $i=1,2$, of positive integers are given by
$U_1=F_1$, $U_2=(F_2)_\Downarrow$. While for $s_{F_2}> -\hat c$ we have
$$
\nu_{\hat c,\F}^a=\nu_{0,\U}^a,
$$
where $\U=(U_1,U_2)$, and the finite sets $U_i$, $i=1,2$, of nonnegative integers are given by
$$
U_1=F_1,\quad U_2=\{0,1,\cdots, s_{F_2}-1+\hat c\}\cup (s_{F_2}+\hat c+(F_2)_\Downarrow).
$$
\end{remark}

\begin{remark}\label{rm1} Contrary to the previous remarks, for $\hat c\in\{-1,-2,\cdots\}$ and $F_1,F_2$ finite sets of positive integers, we show that the measure $\nu ^{a}_{\hat c;\F}$ is never of the form
$$
C\rho_{b,d}^\U(x-s),
$$
for any pair $\U=(U_1,U_2)$ of finite sets of positive integers and any numbers $b, d, s, C$, with $0<b<1$, $d\not \in \{0,-1,-2,\cdots \}$ and $s\in \ZZ$.  Assume on the contrary that
$\nu ^{a}_{\hat c;\F}=C\rho_{b,d}^\U(x-s)$.  Since $0\not \in F_1$, we deduce that $0\not \in \Hh$, and then the first point in the support of $\nu ^{a}_{\hat c;\F}$ is $0$. Since also the first point in the support of $\rho_{b,d}^\U$ is $0$, we have that $s=0$. Trivially $b=a$ and $d \in \{1,2,\cdots \}$, and so
$$
\prod _{h\in \Hh}(x-h)=(x+d-1)\cdots (x+1)\prod_{f\in U_1}(x-f)\prod_{f\in U_2}(x+d+f).
$$
From where we get that $U_1=\{h:h\in \Hh, h\ge 1\}$. Since $-\hat c\in F_1$ (it is an easy consequence of (\ref{hf2}) taking into account that $0\not \in F_2$), we conclude that $-\hat c$ is not in the support of the measure $\nu ^{a}_{\hat c;\F}$. Since we are assuming $\nu ^{a}_{\hat c;\F}=C\rho_{a,d}^\U$, we conclude that $-\hat c\in U_1$. This gives $-\hat c\in \Hh$ (because $-\hat c\ge 1$). But since $0\not \in F_2$, we have $-\hat c\not \in -\hat c-F_2$. The definition of $\Hh$ (\ref{defH}) finally implies $-\hat c\not \in \Hh$.
\end{remark}

\begin{remark}\label{rm4} We finally show that when the parameter $d$ is a positive integer, the measures obtained from the Meixner measure $\rho _{a,d}$ by removing a finite number of their mass points can be written in the form $\nu ^{a}_{\hat c;\F}$, for certain $\hat c\in \{-1,-2,\cdots \}$ and certain pair $\F=(F_1,F_2)$ of finite sets of nonnegative integers (more precisely: we always have $0\not \in F_2$, but $0\not \in F_1$ if and only if the mass point at zero is not removed from the Meixner measure). Indeed, let $A$ be a finite set of nonnegative integers, and consider the measure
$$
\sum_{x\in \NN\setminus A}\frac{a^x\Gamma(x+d)}{x!}\delta_x
$$
obtained by removing the masses at $A$ from the Meixner weight $\rho _{a,d}$, with $0<\vert a\vert <1$, $d\in \{1,2,\cdots \}$.

Write $\hat c=-\max A$, $A^c=\{b:1\le b\le -\hat c, b\not \in A\}$,
$$
X=\begin{cases}(-\hat c-A^c)\cup \{-\hat c\},& 0\not \in A,\\
(-\hat c-A^c),& 0\in A,
\end{cases}
$$
$Y=\{-\hat c+i, i=1,\cdots , d-1\}$, and finally $F_1=A$, $F_2=X\cup Y$.

From the definition of $A^c$, we easily get that the elements of $X$ are positive integers and that $\max X\le -\hat c$. Since $\min Y=-\hat c+1$, this shows that $X\cap Y=\emptyset$. We also have
\begin{align*}
-\hat c -F_2&=-\hat c -(X\cup Y)=(-\hat c -X)\cup (-\hat c-Y)\\
&=\begin{cases} A^c\cup \{-d+1,-d+2,\cdots, -1,0\}, &0\not \in A,\\
A^c\cup \{-d+1,-d+2,\cdots, -1\},& 0\in A.
\end{cases}
\end{align*}
Hence $F_1\cup (-\hat c -F_2)=\{-d+1,-d+2,\cdots, -1,0,1,2,\cdots,-\hat c\}$ and $F_1\cap (-\hat c -F_2)=\emptyset$. From where we get that (\ref{hf2}) holds and that
$$
\Hh=[[F_1\cup (-\hat c -F_2)]\setminus \{0,1,2,\cdots,-\hat c\}]\cup [F_1\cap (-\hat c -F_2)]=\{-d+1,-d+2,\cdots, -1\}.
$$
So
$$
\frac{\Gamma(x+d)}{x!}=(x+d-1)\cdots (x+1)=\prod_{h\in \Hh}(x-h).
$$
\end{remark}

\medskip

In the rest of this Section, we explicitly construct orthogonal polynomials with respect to the measure $\nu ^{a}_{\hat c;\F}$ and prove that they are also eigenfunctions of a higher order difference operator of the form (\ref{hodo}).

Given a pair $\F=(F_1,F_2)$ of finite sets of positive integers, $F_i$ with $k_i$ elements and $k=k_1+k_2$, the polynomial $\Omega _\F^{a,c}$ is defined by
\begin{equation}\label{defomegm}
\Omega _\F^{a,c}(x)=  \left|
  \begin{array}{@{}c@{}lccc@{}c@{}}
  & &&\hspace{-.9cm}{}_{1\le j\le k} \\
  \dosfilas{m_{f}^{a,c}(x+j-1) }{f\in F_1} \\
    \dosfilas{m_{f}^{1/a,c}(x+j-1)/a^{j-1}}{f\in F_2}
      \end{array}
  \right|,
\end{equation}
where the elements of $F_i$ are arranged in increasing order.

Throughout the rest of this paper, we use the following notation:
given a finite set of positive integers $F=\{f_1,\ldots , f_k\}$, the expression
\begin{equation}\label{defdosf}
  \begin{array}{@{}c@{}lccc@{}c@{}}
  &  &&\hspace{-.9cm}{}_{1\le j\le k} \\
    \dosfilas{ z_{f,j}  }{f\in F}
  \end{array}
\end{equation}
inside  of a matrix or a determinant will mean the submatrix defined by
$$
\left(
\begin{array}{cccc}
z_{f_1,1} & z_{f_1,2} &\cdots  & z_{f_1,k}\\
\vdots &\vdots &\ddots &\vdots \\
z_{f_k,1} & z_{f_k,2} &\cdots  & z_{f_k,k}
\end{array}
\right) .
$$
The determinant (\ref{defomegm}) should be understood in this form.

Using Lemma 3.4 of \cite{DdI}, we have that $\Omega _\F^{a,c}$ is a polynomial of degree $\sum_{f\in F_1}f+\sum_{f\in F_2}f-\binom{k_1}{2}-\binom{k_2}{2}$.

\begin{theorem}\label{1th} Given a negative integer $\hat c\in \{-1,-2,\cdots \}$,
let $\F=(F_1,F_2)$ be a pair of finite sets of positive integers, $F_i$ with $k_i$ elements, respectively, satisfying (\ref{hf2}). Denote as before $\Hh$ for the finite set of integers
$$
\Hh=[[F_1\cup (-\hat c-F_2)]\setminus \{0,1,2,\cdots , -\hat c\}]\cup [F_1\cap (-\hat c-F_2)].
$$
Write $c=\hat c+\max F_1+\max F_2+2$, $G_i$ for the finite set of positive integers $G_i=I(F_i)$, $i=1,2$, where $I$ is the involution defined in (\ref{dinv}), and $m_i$ for the number of elements of $G_i$. Consider the sequence of polynomials defined by
\begin{equation}\label{qusmei}
q_n^{a,c;\F}(x)=
\left|
  \begin{array}{@{}c@{}lccc@{}c@{}}
  & m_{n-j+1}^{a,c}(x-\max F_1-1)/(a-1)^{j-1} &&\hspace{-.6cm}{}_{1\le j\le m+1} \\
    \dosfilas{m_{g}^{a,2- c}(-n+j-2) }{g\in G_1} \\
    \dosfilas{m_{g}^{1/a,2- c}(-n+j-2)/a^{j-1}}{g\in G_2}
  \end{array}\right| ,
\end{equation}
where $m=m_1+m_2$. Assume that
\begin{equation}\label{hf1}
\Omega ^{a,\hat c}_\F(n)\not =0,\quad n=0,1,2,\cdots .
\end{equation}
Then the polynomials $(q_n^{a,c;\F})_n$ are eigenfunctions of a difference operator of the form (\ref{hodo}) with $-s=r=1+\sum_{i=1}^2\left(\sum_{f\in F_i}f-\binom{k_i}{2}\right)$, and are also orthogonal with respect to the measure
\begin{equation}\label{defnu2}
\nu ^{a}_{\hat c;\F}=\sum _{x=0;x\not \in F_1}^\infty \prod _{h\in \Hh}(x-h)a^x\delta_x,
\end{equation}
with norm equal to
\begin{equation}\label{nm2}
\langle q_n^{a,c;\F},q_n^{a,c;\F}\rangle=C_\F\frac{a^{n}\Gamma(n+k+\hat c)}{(1-a)^{2n+\hat c} n!}\Omega ^{a,\hat c}_\F(n)\Omega ^{a,\hat c}_\F(n+1),
\end{equation}
where $C_\F$ is a positive constant depending only on $a$ and $\F$.
\end{theorem}

\begin{proof}
We need the invariance for quasi Casorati-Meixner determinants (see Theorem 5.1 of \cite{ducu}), which we write in the form
\begin{equation}\label{duc}
\Omega^{a,d}_\F(x)=(-1)^wa^u(1-a)^v\Omega^{a,-d-\max F_1-\max F_2}_\G(-x),
\end{equation}
where $\G=I(\F)$, and  $w,u,v$ are certain integers which only depend on the pair $\F$.

The identity (\ref{duc}) for $d=\hat c$ shows that the assumption (\ref{hf1}) is equivalent to
\begin{equation}\label{hf12}
\Omega ^{a,2-c}_\G(-n)\not =0,\quad n=0,1,2,\cdots
\end{equation}
(because $-\hat c-\max F_1-\max F_2=2-c$). Hence, we conclude that $q_n^{a,c;\F}$ is a polynomial of degree $n$.

Notice also that since $\F$ satisfies (\ref{hf2}), $c$ has to be a positive integer.

That the polynomials $(q_n^{a,c;\F})_n$ are eigenfunctions of a difference operator of the form (\ref{hodo}) is a direct consequence of Theorem 3.2 of \cite{DdI} (after a suitable re-normalization of the polynomials) and the formulas for the $\D$-operators for the Meixner polynomials displayed in the Section 6 of \cite{DdI} (see also Section 5 of \cite{du1}).

Since $\Omega ^{a,2-c}_\F$ is a polynomial, using (\ref{hf12}) and the Hurwitz Theorem,
we can find a sequence of real numbers $c_s\not\in \ZZ$ with $\lim_{s\to \infty}c_s=c$ and such that for $s\in \NN$
\begin{equation}\label{hf1n}
\Omega^{a,2-c_s}_\F (-n)\not =0,\quad n=0,1,2,\cdots
\end{equation}
Write $\hat c_s=-c_s+\max F_1+\max F_2+2\not \in \ZZ$.

According to Theorem 6.2 of \cite{DdI}, for $s\in \NN$, the polynomials
\begin{equation}\label{qusmen}
q_n^{a,c_s;\F}(x)=
\left|
  \begin{array}{@{}c@{}lccc@{}c@{}}
  & m_{n-j+1}^{a,c_s}(x-\max F_1-1)/(a-1)^{j-1} &&\hspace{-.6cm}{}_{1\le j\le m+1} \\
    \dosfilas{m_{g}^{a,2- c_s}(-n+j-2) }{g\in G_1} \\
    \dosfilas{m_{g}^{1/a,2- c_s}(-n+j-2)/a^{j-1}}{g\in G_2}
  \end{array}\right|
\end{equation}
are orthogonal with respect to the measure
$$
\rho_{a,\hat c_s}^\F=\sum _{x=0}^\infty \prod _{f\in F_1}(x-f)\prod _{f\in F_2}(x+\hat c_s+f)\frac{a^x\Gamma(x+\hat c_s)}{x!}\delta_x.
$$
Moreover
$$
\langle q_n^{a,c_s;\F},q_n^{a,c_s;\F}\rangle _{\rho_{a,\hat c_s}^\F}=\frac{a^{n+k_1-2m_2}\Gamma(n+k+\hat c_s)}{(1-a)^{2n+k+\hat c_s} n!}\Omega ^{a,2-c_s}_\G(-n)\Omega ^{a,2-c_s}_\G(-n-1)
$$
(this identity is implicit in the identities (5.8), (5.9) and (5.10), and Lemma 4.2 of
\cite{DdI} and it can be proved as the identities (3.17) and (3.19) of \cite{ducu}).

Using (\ref{duc}), this can be written
\begin{equation}\label{nmqt}
\langle q_n^{a,c_s;\F},q_n^{a,c_s;\F}\rangle_{\rho_{a,\hat c_s}^\F}=C_\F\frac{a^{n}\Gamma(n+k+\hat c_s)}{(1-a)^{2n+\hat c_s} n!}\Omega ^{a,\hat c_s}_\F(n)\Omega ^{a,\hat c_s}_\F(n+1),
\end{equation}
where $C_\F=a^{2u+k_1}(1-a)^{2v-k}>0$ and only depends on $a$ and $\F$.
Notice that from the hypothesis (\ref{hf2}), it follows that $1\le \hat c+k$.
Taking then into account that
$$
\lim_{s\to+\infty}q_n^{a,c_s;\F}(x)=q_n^{a,c;\F}(x),\quad \lim_{s\to+\infty}\Omega^{a,c_s}_\F(x)=\Omega^{a,c}_\F(x),
$$
the Theorem is an easy consequence of the Lemmas \ref{1le} and \ref{lemi}.

\end{proof}

In the next Section, see Corollary \ref{eles}, we prove that when the measure $\nu ^{a,\hat c}_\F$ is positive, then the assumption (\ref{hf1}) holds and we can construct a sequence of orthogonal polynomials with respect to $\nu ^{a,\hat c}_\F$ using (\ref{qusmei}).

\section{New exceptional Meixner orthogonal polynomials}\label{sect5}
As in \cite{dume}, for real numbers $a,\hat c$, with $a\not = 0,1$, we associate to each pair $\F=(F_1,F_2)$ of finite sets of positive integers the polynomials $m_n^{a,\hat c;\F}$, $n\in \sigma_\F$, defined as follows
\begin{equation}\label{defmex}
m_n^{a,\hat c;\F}(x)=  \left|
  \begin{array}{@{}c@{}cccc@{}c@{}}
    & m_{n-u_\F}^{a,\hat c}(x+j-1)& &\hspace{-.6cm}{}_{1\le j\le k+1} \\
    \dosfilas{ m_{f}^{a,\hat c}(x+j-1)  }{f\in F_1} \\
    \dosfilas{ m_{f}^{1/a,\hat c}(x+j-1)/a^{j-1} }{f\in F_2}
  \end{array}
  \right|
\end{equation}
where the number $u_\F$ and the infinite set of nonnegative integers $\sigma _\F$ are defined by (\ref{defuf}) and (\ref{defsf}), respectively. The determinant (\ref{defmex}) should be understood as explained in (\ref{defdosf}). Using Lemma 3.4 of \cite{DdI}, we deduce that $m_n^{a,\hat c;\F}$, $n\in \sigma _\F$, is a polynomial of degree $n$ with leading coefficient equal to
\begin{equation}\label{lcrn}
(-1)^{k_2(k_1+1)} \frac{(a-1)^{k_2(k_1+1)}V_{F_1}V_{F_2}\prod_{f\in F_1}(f-n+u_\F)}{a^{k_2k_1+\binom{k_2+1}{2}}(n-u_\F)!\prod_{f\in F_1}f!\prod_{f\in F_2}f!},
\end{equation}
where $V_F$ is the Vandermonde determinant (\ref{defvdm}).

Combining columns in (\ref{defmex}), we have the alternative definition
\begin{equation}\label{defmexa}
m_n^{a,c;\F}(x)= \left|
  \begin{array}{@{}c@{}cccc@{}c@{}}
    & m_{n-u_\F-j+1}^{a,\hat c+j-1}(x)& &\hspace{-.6cm}{}_{1\le j\le k+1} \\
    \dosfilas{ m_{f-j+1}^{a,\hat c+j-1}(x)  }{f\in F_1} \\
    \dosfilas{\frac{(1-a)^{j-1}}{a^{j-1}} m_{f}^{1/a,\hat c+j-1}(x)}{f\in F_2}
  \end{array}
  \right| .
\end{equation}

Taking a sequence of numbers $\hat c_s\not =0,-1,-2,\cdots $ with $\lim_{s\to+\infty}\hat c_s=\hat c$, we can easily prove using Theorem 3.3 of \cite{dume} that the polynomials $m_n^{a,\hat c;\F}$ are eigenfunctions of a second order difference operator with rational coefficients.

\begin{theorem}\label{th3.3} Let $\F=(F_1,F_2)$ be a pair of finite sets of positive integers. Then the polynomials $m_n^{a,\hat c;\F}$ (\ref{defmex}), $n\in \sigma _\F$,
are common eigenfunctions of the second order difference operator
\begin{equation}\label{sodomex}
D_\F=h_{-1}(x)\Sh_{-1}+h_0(x)\Sh_0+h_1(x)\Sh_{1},
\end{equation}
where
\begin{align}\label{jpm1}
h_{-1}(x)&=\frac{x\Omega^{a,\hat c}_\F(x+1)}{(a-1)\Omega^{a,\hat c}_\F(x)},\\\label{jpm2}
h_0(x)&=-\frac{(1+a)(x+k)+a\hat c}{a-1}+u_\F+\Delta\left(\frac{a(x+\hat c+k-1)\Lambda^{a,\hat c}_\F(x)}{(a-1)\Omega^{a,\hat c}_\F(x)}\right),\\\label{jpm3}
h_1(x)&=\frac{a(x+\hat c+k)\Omega^{a,\hat c}_\F(x)}{(a-1)\Omega^{a,\hat c}_\F(x+1)},
\end{align}
$\Delta $ denotes the first order difference operator $\Delta f=f(x+1)-f(x)$
and
$$
\Lambda_\F^{a,\hat c}(x)=\left|
  \begin{array}{@{}c@{}ccccc@{}c@{}}
    \dosfilas{ m_{f}^{a,\hat c}(x) & m_{f}^{a,\hat c}(x+1) &\cdots  & m_{f}^{a,\hat c}(x+k-2)& m_{f}^{a,\hat c}(x+k) }{f\in F_1} \\
    \dosfilas{ m_{f}^{1/a,\hat c}(x) & \displaystyle \frac{m_{f}^{1/a,\hat c}(x+1)}{a} & \cdots & \displaystyle \frac{m_{f}^{1/a,\hat c}(x+k-2)}{a^{k-2}} & \displaystyle \frac{m_{f}^{1/a,\hat c}(x+k)}{a^{k}} }{f\in F_2}
  \end{array}
  \right|.
$$
Moreover $D_\F(m_n^{a,\hat c;\F})=nm_n^{a,\hat c;\F}$, $n\in \sigma_\F$.
\end{theorem}

The polynomials $(m_n^{a,\hat c;\F}(x))_{n\in\sigma _\F}$ are called exceptional Meixner polynomials when, in addition, they are orthogonal and complete with respect to a positive measure. This leads to the concept of admissibility.

When $d\not =0,-1,-2,\cdots $, we say in \cite{dume} that $d$ and $\F$ are admissible if for all $x\in \NN $
\begin{equation}\label{defadm}
\prod_{f\in F_1}(x-f)\prod_{f\in F_2}(x+d+f)\Gamma(x+d)\ge 0
\end{equation}
(see the Definition 2.5 of \cite{dume}). In other words, for $0<a<1$, the measure
\begin{equation}\label{mraf}
\rho _{a,d}^{\F}=\sum _{x=0}^\infty \prod_{f\in F_1}(x-f)\prod_{f\in F_2}(x+d+f)\frac{a^{x}\Gamma(x+d)}{x!}\delta _x
\end{equation}
is positive.

Since the measure $\rho_{a,d}^\F$ is a Christoffel transform of the Meixner measure, using \cite[Theorem 2.5]{Sz} we can construct orthogonal polynomials with respect to this measure by means of the determinant
\begin{equation}\label{defqnme}
q_n^\F(x)=\frac{\left|
  \begin{array}{@{}c@{}cccc@{}c@{}}
   &m_{n+j-1}^{a,d}(x)& &\hspace{-.6cm}{}_{1\le j\le k+1}\\
    \dosfilas{m_{n+j-1}^{a,d}(f) }{f\in F_1} \\
    \dosfilas{m_{n+j-1}^{1/a,d}(f) }{f\in F_2}
  \end{array}
  \right|}{\prod_{f\in F_1}(x-f)\prod_{f\in F_2}(x+d+f)}.
\end{equation}
In \cite{dume} and under the assumption $d\not =0,-1,-2,\cdots $ and $\rho _{a,d}^{\F}$ being a positive measure, we constructed exceptional Meixner polynomials by using as the key tool the duality between the polynomials $(m_n^{a,d;\F}(x))_{n\in\sigma _\F}$ and $(q_n^\F(x))_n$.
To do that, in Lemma 4.2 of \cite{dume}, we prove the following equivalent condition to the admissibility of $d$ and $\F$.

\begin{lemma}[Lemma 4.2 of \cite{dume}]\label{l3.1} Given  real numbers $a,d$, with $0<a<1$ and $d\not =0,-1,-2,\ldots$, and a pair $\F $ of finite sets of positive integers, the following conditions are equivalent.
\begin{enumerate}
\item The measure $\rho_{a,d}^\F$ (\ref{mqs}) is positive.
\item $d$ and $\F$ are admissible.
\item $\Gamma(n+d+k)\Omega_\F ^{a,d}(n)\Omega_\F ^{a,d}(n+1)>0$ for all nonnegative integer $n$, where the polynomial $\Omega_\F^{a,d}$ is defined by (\ref{defomegm}).
\end{enumerate}
\end{lemma}

But when $d=\hat c\in \{-1,-2,\cdots \}$ the measure $\rho_{a,\hat c}^\F$ is not well defined. Under the assumption (\ref{hf2}), this measure might
be substituted by the measure $\nu^{a}_{\hat c;\F}$ (\ref{defnu}) (which, according to the Lemma \ref{1le}, it is the limit of the measures $\rho_{a,\hat c_s}^\F$ when $\hat c_s\to \hat c$). But when $\hat c\in \{-1,-2,\cdots \}$, the orthogonal polynomials with respect to $\nu^{a}_{\hat c;\F}$ can not be represented as in (\ref{defqnme}) because (as explained in the Remark \ref{rm1}) $\nu_{a,\hat c}^\F$ is not anymore the Christoffel transform of a Meixner measure. Hence, to construct the new families of exceptional Meixner polynomials we will avoid the duality and proceed in a different way (with respect to the approach we used in \cite{dume}).

In the rest of this Section we will assume $\hat c\in \{-1,-2,\cdots \}$, $\F=(F_1,F_2)$ where $F_i$ is a finite set of positive integers  and
\begin{equation}\label{hf2s}
\{0,1,2,\cdots , -\hat c\}\subset F_1\cup (-\hat c-F_2).
\end{equation}
As in the previous Section (see Remarks \ref{remp}, \ref{rems}), we do not need to consider neither the case $\hat c=0$ nor $0\in F_1\cup F_2$. Indeed, it is not difficult to check that in those cases and up to a multiplicative constant the polynomial
$m_n^{a,\hat c;\F}$ is equal to $m_n^{a,d;\U}$, where $d=\hat c+s_1+s_2$ and $\U=(U_1,U_2)$ and
\begin{equation}\label{cmlp}
s_i=\begin{cases} 0, &0\not \in F_i,\\ s_{F_i},& 0\in F_i,\end{cases},\quad U_i=\begin{cases} F_i, &0\not \in F_i,\\ (F_i)_{\Downarrow},& 0\in F_i\end{cases}
\end{equation}
(see (\ref{defs0}) and (\ref{defffd})).

When $d=\hat c \in \{-1,-2,\cdots \}$, we define the admissibility condition for $\hat c$ and $\F$ as the positivity of the measure $\nu^{a}_{\hat c;\F}$ (\ref{defnu}).

\begin{definition}\label{admp}
For $\hat c\in \{-1,-2,\cdots \}$ and pair $\F$ of set of positive integers satisfying (\ref{hf2s}), we say that $\hat c$ and $\F$ are admissible if for all $x\in \NN \setminus F_1$
\begin{equation}\label{defadm2}
\prod_{h\in \Hh}(x-h)\ge 0,
\end{equation}
where the finite set $\Hh$ is defined by (\ref{defH})
\end{definition}

Using Lemma \ref{1le}, we can see that the admissibility condition (\ref{defadm2}) follows from (\ref{defadm}) by changing $\hat c$ to $\hat c_s$ and then taking limit when $\hat c_s$ goes to $\hat c$.

We next prove the following analogous to the Lemma \ref{l3.1}.

\begin{lemma}\label{l3.1p} Let $a,\hat c$ be real numbers with $0<a<1$ and $\hat c\in \{-1,-2,\ldots\}$, and let $\F $ be a pair of finite sets of positive integers satisfying (\ref{hf2s}). Then the following conditions are equivalent.
\begin{enumerate}
\item The measure $\nu^{a}_{\hat c;\F}$ (\ref{defnu}) is positive.
\item $\hat c$ and $\F$ are admissible (according to the Definition \ref{admp}).
\item $\Omega_\F ^{a,\hat c}(n)\Omega_\F ^{a,\hat c}(n+1)>0$ for all nonnegative integer $n$, where the polynomial $\Omega_\F^{a,c}$ is defined by (\ref{defomegm}).
\end{enumerate}
\end{lemma}

\begin{proof}

Notice that (\ref{hf2s}) implies that $\hat c+k\ge 1$ and then $\Gamma(n+\hat c+k)>0$, which it shows the analogy between the condition (3) in the Lemmas \ref{l3.1} and \ref{l3.1p}.

The equivalente between (1) and (2) is a direct consequence of the Definition \ref{admp}.

We next prove that (1) implies (3).

We begin by proving that $\Omega_\F ^{a,\hat c}(n)\not =0$ for all nonnegative integer $n$. Using Lemma \ref{1le}, we can take a sequence $\hat c_s\not \in \{0,-1,-2,\cdots \}$ with  $\lim _{s\to \infty}\hat c_s=\hat c$ and such that the measures
$\rho ^{a,\hat c_s}_\F$ are positive. According to Lemma \ref{l3.1}, this means that $\hat c_s$ and $\F$ are admissible (according to (\ref{defadm})). Write $c_s=\hat c_s+\max F_1+\max F_2+2$ (like in the Theorem \ref{1th}) and $\G=I(\F)$. From (\ref{qusmen}), we have that for $n\in \NN$
$$
Q_n^{a,c_s;\F}(x)=\frac{1}{\Omega^{a,2-c_s}_{\G}(-n)} q_n^{a,c_s;\F}(x),
$$
is the $n$-th monic orthogonal polynomial with respect to the positive measure $\rho_{a,\hat c_s}^\F$ (\ref{mqs}). Using (\ref{nmqt}) and (\ref{duc}), we get that
\begin{equation}\label{phx5}
\Vert Q_n^{a,c_s;\F}\Vert_2^2=C_\F\frac{a^{n}\Gamma(n+k+\hat c_s)\Omega ^{a,\hat c_s}_\F(n+1)}{(1-a)^{2n+\hat c_s} n!\Omega ^{a,\hat c_s}_\F(n)}.
\end{equation}
Write $Q_n^{a,c;\F}$ for the $n$-th monic orthogonal polynomial with respect to the positive measure $\nu^{a}_{\hat c;\F}$, $c=\hat c+\max F_1+\max F_2+2$. Since the positive measures $\rho_{a,\hat c_s}^\F$ converge to $\nu^{a}_{\hat c;\F}$ as $\hat c_s$ goes to $\hat c$, we get that the norm for the $n$-th monic orthogonal polynomial $Q_n^{a,c_s;\F}$ has to converge to the norm of $Q_n^{a,c;\F}$. Since $\lim_{s\to \infty}\Omega ^{a,\hat c_s}_\F=\Omega ^{a,\hat c}_\F$, according to (\ref{phx5}), we conclude that if for some $n_0$, $\Omega ^{a,\hat c}_\F(n_0)=0$ then $\Omega ^{a,\hat c}_\F(n)=0$, for $n\ge n_0$. Which it is a contradiction because $\Omega ^{a,\hat c}_\F$ is a non-null polynomial.

Since, we have already proved that $\Omega_\F ^{a,\hat c}(n)\not =0$ for all nonnegative integer $n$, Theorem \ref{1th} shows that $q_n^{a,\hat c;\F}$, is the $n$-th orthogonal polynomial with respect to the positive measure $\nu^{a}_{\hat c;\F}$. Hence, it has positive norm and so the condition (3) follows from (\ref{nm2}).

We finally prove that (3) implies (1). From the condition (3), we deduce that for $n\ge 0$, $\Omega^{a,\hat c}_\F(n)\not =0$ and $\Omega^{a,\hat c}_\F(z)$ has an even number of zeros in each interval $(n,n+1)$ (counting multiplicities), $n\in \NN$. Since $\Omega^{a,\hat c}_\F(z)$ is a polynomial, using Hurwitz theorem, we can take $\epsilon >0$ and a sequence $\hat c_s\not \in \{0,-1,-2,\cdots \}$ with  $\lim _{s\to \infty}\hat c_s=\hat c$ and such that for each $n\ge 0$, $\Omega^{a,\hat c_s}_\F(n)\not =0$ and $\Omega^{a,\hat c_s}_\F(z)$ has in $U_n=\{z:n< \Re z< n+1, \vert \Im z \vert <\epsilon\}$ the same number of zeros as $\Omega^{a,\hat c}_\F(z)$, that is, an even number of zeros. Since $\Omega^{a,\hat c_s}_\F(x)\in \RR$, for $x\in \RR$, we get that $\Omega^{a,\hat c_s}_\F(z)$ has in $U_n$ an even number of non-real zeros, and then has also an even number of real zeros. So, $\Omega^{a,\hat c_s}_\F(n)\Omega^{a,\hat c_s}_\F(n+1)>0$, $n\ge 0$. From (\ref{hf2s}), we have that $\hat c+k\ge 1$ and then we can also assume that
$\hat c_s+k>0$, and hence $\Gamma (n+\hat c_s+k)\Omega^{a,\hat c_s}_\F(n)\Omega^{a,\hat c_s}_\F(n+1)>0$, $n\ge 0$. Using Lemma \ref{l3.1}, we conclude that $\hat c_s$ and $\F$ are admissible (according to (\ref{defadm})). That is the measure $\rho^{a,\hat c_s}_\F$ is positive. Hence, the positivity of the measure $\nu^{a}_{\hat c;\F}$ follows from the Lemma \ref{1le}.

\end{proof}

As a consequence we have.

\begin{corollary}\label{eles} Let $a,\hat c$ be real numbers with $0<a<1$ and $\hat c\in \{-1,-2,\ldots\}$, and let $\F $ be a pair of finite sets of positive integers satisfying (\ref{hf2s}). If the measure $\nu^{a}_{\hat c;\F}$ (\ref{defnu}) is positive, then
$\Omega_\F ^{a,\hat c}(n)\not=0$, $n\ge 0$, and then we can construct orthogonal polynomials $(q_n^{a,c;\F})_n$ with respect to $\nu ^{a}_{\hat c;\F}$ using (\ref{qusmei}).
\end{corollary}

The orthogonality and completeness of the exceptional Meixner polynomials (\ref{defmex}) under the admissibility of $\hat c$ and $\F$ is then a consequence of Theorems 4.3 and 4.4 in \cite{dume}.

\begin{theorem}\label{2th} Given  real numbers $a$ and $\hat c$, with $0<a<1$ and $\hat c\in \{-1,-2,\cdots \}$, and a pair $\F $ of finite sets of positive integers satisfying (\ref{hf2s}), assume that $\hat c$ and $\F$ are admissible. Then the  polynomials $m_n^{a,\hat c,\F}$, $n\in \sigma _\F$,
are orthogonal with respect to the positive measure
\begin{equation}\label{momex}
\omega_{a,\hat c}^\F=\sum_{x=0}^\infty \frac{a^x\Gamma(x+\hat c+k)}{x!\Omega_\F^{a,\hat c}(x)\Omega_\F^{a,\hat c}(x+1)}\delta_x,
\end{equation}
with
$$
\Vert m_n^{a,\hat c,\F}\Vert _2^2=\frac{a^{n-u_\F +k_1-2k}}{(1-a)^{\hat c+2n-2u_\F-k}}\prod_{h\in \Hh}(n-u_\F-h),
$$
where the finite set of integers $\Hh$ is defined by (\ref{defH}).
Moreover the linear combinations of the  polynomials $m_n^{a,\hat c;\F}$, $n\in \sigma _\F$, are dense in $L^2(\omega_{a,\hat c}^{\F})$.
\end{theorem}

\begin{proof}
The Lemma \ref{l3.1p} says that the admissibility of $\hat c$ and $\F$ is equivalent to
$$
\Omega_\F^{a,\hat c}(x)\Omega_\F^{a,\hat c}(x+1)>0, \quad x\in \NN.
$$
We can then take as before a sequence of numbers $\hat c_s\not\in \{0,-1,-2,\cdots \}$ with $\lim_{s\to \infty}\hat c_s=\hat c$ and such that
$$
\Omega_\F^{a,\hat c_s}(x)\Omega_\F^{a,\hat c_s}(x+1)>0, \quad x\in \NN.
$$
Since
$$
\lim_{s\to \infty}\Omega^{a,\hat c_s}_\F(z)=\Omega^{a,\hat c}_\F(z),
$$
uniformly in compact set of $\CC$, we get that there exists $M>0$ such that $\vert \Omega^{a,\hat c_s}_\F(x)\vert\ge M$, $x\in \NN$. Hence
$$
\vert \omega_{a,\hat c_s}^\F(x)\vert \le \frac{a^x\Gamma(x+\hat c+k)}{M^2x!}.
$$
The Theorem follows easily from the
Theorems 4.3 and 4.4 in \cite{dume}, and the Lemma \ref{lemi} in the Preliminaries.

\end{proof}

\section{New exceptional Laguerre orthogonal polynomials}\label{sect6}
As in \cite{dume} we can construct exceptional Laguerre polynomials by taking limit as $a$ goes to $1$ in the exceptional Meixner polynomials constructed in the previous Section using the basic limit
(\ref{blmel}).

Given a negative integer $\hat \alpha\in \{-1,-2,\cdots \}$ and a pair $\F=(F_1,F_2)$ of finite sets of positive integers, writing $\hat c=\hat \alpha +1$, using the expression (\ref{defmexa}) for the polynomials $m_n^{a,\hat c;\F}$, $n\in\sigma_\F$, setting $x\to x/(1-a)$ and  taking limit as $a\to 1$, we get (up to normalization constants) the polynomials, $n\in \sigma _\F$,
\begin{equation}\label{deflax}
L_n^{\hat \alpha ;\F}(x)= \left|
  \begin{array}{@{}c@{}cccc@{}c@{}}
    & (L_{n-u_\F}^{\hat \alpha})^{(j-1)}(x) & &\hspace{-.6cm}{}_{1\le j\le k+1} \\
    \dosfilas{(L_{f}^{\hat \alpha})^{(j-1)}(x) }{f\in F_1} \\
    \dosfilas{L_{f}^{\hat \alpha +j-1}(-x) }{f\in F_2}
  \end{array}
  \right|.
\end{equation}
More precisely
\begin{equation}\label{lim1}
\lim_{a\to 1}(a-1)^{n-(k_1+1)k_2}m_n^{a,\hat c;\F}\left(\frac{x}{1-a}\right)=(-1)^{\binom{k+1}{2}+\sum_{f\in F_2}f}L_n^{\hat \alpha ;\F}(x)
\end{equation}
uniformly in compact sets.

Notice that $L_n^{\hat \alpha ;\F}$ is a polynomial of degree $n$ with leading coefficient equal to
$$
(-1)^{n-u_\F+\sum_{f\in F_1}f} \frac{V_{F_1}V_{F_2}\prod_{f\in F_1}(f-n+u_\F)}{(n-u_\F)!\prod_{f\in F_1}f!\prod_{f\in F_2}f!},
$$
where the integer $u_\F$ is defined in (\ref{defuf}) and $V_F$ is the Vandermonde determinant defined by (\ref{defvdm}).

We define the associated polynomial
\begin{equation}\label{defhom}
\Omega _{\F}^{\hat \alpha}(x)=
\left|
  \begin{array}{@{}c@{}cccc@{}c@{}}
    &  & &\hspace{-.9cm}{}_{1\le j\le k} \\
    \dosfilas{(L_{f}^{\hat \alpha})^{(j-1)}(x) }{f\in F_1} \\
    \dosfilas{L_{f}^{\hat \alpha +j-1}(-x) }{f\in F_2}
  \end{array}
  \right|.
\end{equation}
$\Omega_{\F}^{\hat \alpha}$ is a polynomials of degree $u_\F+k_1$.

The polynomials $L_n^{\hat \alpha ;\F}$, $n\in \sigma_\F$, are eigenfunctions of a second order differential operator.

\begin{theorem}\label{th5.1} Given a real number $\hat \alpha=-1,-2,\cdots $ and a pair $\F$ of finite sets of positive integers, the polynomials $L_n^{\alpha;\F}$, $n\in \sigma _\F$,
are common eigenfunctions of the second order differential operator
\begin{equation}\label{sodolax}
D_F=x\partial ^2+h_1(x)\partial+h_0(x),
\end{equation}
where $\partial=d/dx$ and
\begin{align}\label{jph1}
h_1(x)&=\alpha +k+1-x-2x\frac{(\Omega_\F^{\hat \alpha})'(x)}{\Omega_\F^{\hat \alpha}(x)},\\\label{jph2}
h_0(x)&=-k_1-u_\F +(x-\alpha -k)\frac{(\Omega_\F^{\hat \alpha})'(x)}{\Omega_\F^{\hat \alpha}(x)}+x\frac{(\Omega_\F^{\hat \alpha})''(x)}{\Omega_\F^{\hat \alpha}(x)}.
\end{align}
More precisely $D_\F(L_n^{\alpha;\F})=-nL_n^{\alpha;\F}(x)$.
\end{theorem}

In the rest of this Section we will assume that $\hat \alpha\in \{-2,-3,\cdots \}$ and
\begin{equation}\label{hf2l}
\{0,1,2,\cdots , -\hat \alpha-1\}\subset F_1\cup (-\hat \alpha-1-F_2).
\end{equation}
We do not need to consider neither the case $\hat \alpha=-1$ nor $0\in F_1\cup F_2$. Indeed, it is not difficult to check that in those cases and up to a sign the polynomial
$L_n^{\hat \alpha;\F}$ is equal to $L_n^{d;\U}$, where $d=\hat \alpha+s_1+s_2$ and $\U=(U_1,U_2)$ are given by (\ref{cmlp}).

We also straightforwardly have
\begin{equation}\label{rromh}
L_{u_\F}^{\alpha ;\F}(x)=(-1)^{\binom{s_\F}{2}+s_\F k_1}\Omega_{\F_\Downarrow }^{\alpha +s_\F}(x),
\end{equation}
where the positive integer $s_\F$ and the pair $\F_\Downarrow$ are defined by (\ref{defs0f}) and (\ref{deffd}), respectively.

The admissibility condition for the new families of exceptional Laguerre polynomials is then the admissibility of $\hat \alpha +1$ and $\F$ according to the Definition \ref{admp}. We need the following Lemma (the proof is similar to that of part (4) in Lemma 2.6 of \cite{dume}, and it is omitted).

\begin{lemma}\label{ppt} Let $\hat \alpha$ and $\F$ be an integer $\hat \alpha\in \{-2,-3,\cdots \}$ and a pair of finite set of positive integers, respectively, satisfying (\ref{hf2l}), and assume that $\hat \alpha+1$ and $\F$ are admissible (according to the Definition \ref{admp}). Then $\hat \alpha +s_\F\not =-1$ and
\begin{enumerate}
\item if $\hat \alpha +s_\F\le -2$ then $\hat \alpha +1+s_\F$ and $\F_\Downarrow$ satisfy (\ref{hf2l}) and are also admissible (according to the Definition \ref{admp});
\item if $\hat \alpha +s_\F\ge 0$ then $\hat \alpha +1+s_\F$ and $\F_\Downarrow$ are also admissible (according to (\ref{defadm})).
\end{enumerate}
\end{lemma}

We are now ready to show that the admissibility of $\hat \alpha +1$ and $\F$ is equivalent to the non existence of nonnegative zeros for the polynomial $\Omega_\F^{\hat \alpha }$.

\begin{theorem} Let $\hat \alpha$ and $\F$ be an integer $\hat \alpha\in \{-2,-3,\cdots \}$ and a pair of finite set of positive integers, respectively, satisfying (\ref{hf2l}). The following conditions are equivalent
\begin{enumerate}
\item $\hat\alpha+1$ and $\F$ are admissible (according to the Definition \ref{admp}).
\item $\Omega_\F^{\hat \alpha }(x)\not =0$, $x\in [0,+\infty)$.
\end{enumerate}
\end{theorem}

\begin{proof}
We fist prove (1) $\Rightarrow $ (2).
Assume on the contrary that there exists $x_0\in [0,+\infty)$ with $\Omega_\F^{\hat \alpha }(x_0)=0$.
Write $\hat c=\hat \alpha +1$. The admissibility of $\hat c$ and $\F$ is equivalent to the positivity of the measure $\nu^{a}_{\hat c;\F}$ (\ref{defnu}) for any $0<a<1$. According to Lemma \ref{1le}, we can take numbers $\hat c_s\not\in \{0,-1,-2,\cdots \}$ with $\lim _{s\to \infty}\hat c_s=\hat c$ and such that the measure $\rho^{a,\hat c_s}_\F$ are also positive. This means that $\hat c_s=\hat \alpha_s+1$ and $\F$ are admissible (according to (\ref{defadm})). Using the Lemma 6.2 and the Corollary 6.4 in \cite{dume} we deduce that $\hat \alpha _s+k>-1$ and $\Omega_\F^{\hat \alpha_s }(x)\not =0$, $x\in [0,+\infty)$, and
$$
\int (L_{n}^{\hat \alpha _s;\F}(x))^2\frac{x^{\hat \alpha _s+k}e^{-x}}{(\Omega_\F^{\hat \alpha_s}(x))^2}dx =a^{-n+u_\F}\rho_\F^{a,\hat c_s}(n-u_\F),\quad n\in \sigma_\F.
$$
Since $\lim_{s\to \infty} L_{n}^{\hat \alpha _s;\F}(z)=L_{n}^{\hat \alpha ;\F}(z)$ and $\lim_{s\to \infty} \Omega_\F^{\hat \alpha_s}(z)=\Omega_\F^{\hat \alpha}(z)$, uniformly in compact set of the complex plane, we conclude that $L_n^{\hat \alpha ;\F}(x_0)=0$, $n\in\sigma _\F$. Using (\ref{rromh}), we get that $\Omega_{\F_\Downarrow}^{\hat \alpha +s_\F}(x_0)=0$. 
Repeating the process using the Lemma \ref{ppt}, after a certain number of steps we will find a nonnegative integer $\alpha\ge 0$ and a pair $\tilde \F$ of finite sets of positive integers such that $\alpha+1$ and $\tilde \F$ are admissible (according to (\ref{defadm})) and $\Omega_{\tilde \F}^{\alpha}(x_0)=0$. But this contradicts \cite[Lemma 6.2]{dume}.

We next prove (2) $\Rightarrow $ (1). From the assumption (\ref{hf2l}), we deduce that $\hat \alpha+k\ge 0$. Since $\Omega_\F^{\hat \alpha }$ is a polynomial, using Hurwitz Theorem we can take a sequence of numbers $\hat \alpha_s\not\in \{-1,-2,\cdots \}$, $\hat \alpha_s+k>-1$ with $\lim _{s\to \infty}\hat \alpha _s=\hat \alpha$ and such that $\Omega_\F^{\hat \alpha_s }(x)\not =0$, $x\in [0,+\infty)$. Using the Theorem 1.2 in \cite{duma}, we conclude that $\hat \alpha_s+1$ and $\F$ are admisible (according to (\ref{defadm})). That is, the measures
$\rho_\F^{a,\hat \alpha _s+1}$ are positive. It is then enough to apply the Lemma \ref{1le}.
\end{proof}

The following Theorem establishes the completeness of the exceptional Laguerre polynomials (the proof is analogous to that of Theorem \ref{2th} and it is omitted).

\begin{theorem}\label{2thL} Given  a negative integer $\hat \alpha\in \{-2,-3,\cdots \}$, and a pair $\F $ of finite sets of positive integers satisfying (\ref{hf2l}), assume that $\hat \alpha+1$ and $\F$ are admissible. Then the  polynomials $L_n^{\hat \alpha,\F}$, $n\in \sigma _\F$,
are orthogonal with respect to the positive weight
\begin{equation}\label{molax}
\omega_{\hat \alpha }^\F(x)=\frac{x^{\hat \alpha +k}e^{-x}}{(\Omega_\F^{\hat \alpha})^2(x)},\quad x\ge 0,
\end{equation}
with
$$
\Vert L_n^{\hat \alpha,\F}\Vert _2^2=\prod_{h\in \Hh}(n-u_\F-h),
$$
where the finite set of integers $\Hh$ is defined by (\ref{defH}) for $\hat c=\hat \alpha +1$.
Moreover the linear combinations of the  polynomials $L_n^{\hat \alpha,\F}$, $n\in \sigma _\F$, are dense in $L^2(\omega_{\hat \alpha }^\F dx)$.
\end{theorem}

The simplest example of Theorem \ref{2thL} is the case $\hat\alpha=-2$, $F_1=F_2=\{1\}$. This gives $\Hh=\emptyset$ and then trivially we have that $\hat \alpha +1$ and $\F$ are admissible. As a consequence, the polynomials
$$
L_n^{\hat \alpha,\F}(x)=\begin{vmatrix} L_{n-1}^{\hat \alpha}(x)&(L_{n-1}^{\hat \alpha})'(x)&(L_{n-1}^{\hat \alpha})''(x)\\
-x-1&-1&0\\
x-1&x&x+1
\end{vmatrix},
$$
$n=1,3,4,5,\cdots $,
are orthogonal and complete with respect to the positive weight
$$
\frac{e^{-x}}{(x^2+1)^2},\quad x\ge 0,
$$
and eigenfunctions of the second order differential operator
$$
x\left(\frac{d}{dx}\right)^2+\left(1-x-\frac{4x^2}{x^2+1}\right)\frac{d}{dx}+\left(-2+\frac{2x+2x^2}{x^2+1}\right).
$$
This is a new example of exceptional Laguerre polynomials because each one of the previous known examples is orthogonal with respect to a weight of the form
$$
\frac{x^\alpha e^{-x} dx}{\tau^2(x)},\quad x\in (0,+\infty ),
$$
where $\alpha\not =0$ and $\tau$ is a polynomial which does not vanish in $[0,+\infty)$ (see \cite{dume} or \cite{BK}).

\section{New Hahn-Krall orthogonal polynomials}\label{sect4}
In this Section we construct new Hahn-Krall orthogonal polynomials. For $a+b+1,a+b+N+1\neq-1,-2,\ldots$ we write $(h_n^{a,b,N})_n$ for the sequence of Hahn polynomials defined by
\begin{equation}\label{HP}
h_n^{a,b,N}(x)=\sum_{j=0}^n\frac{(-x)_j(N-n+1)_{n-j}(a+b+1)_{j+n}(a+j+1)_{n-j}}{(2+a+b+N)_n(n-j)!j!},\quad n\ge 0
\end{equation}
(see \cite{KLS}, pp, 234-7; we have taken a different normalization that in \cite{DdI2} since we are going to deal here with the case when $a\in\{-1,-2,\cdots \}$).

When $N+1$ is a positive integer, $a,b\neq -1,-2,\ldots,-N$, and $a+b\neq -1,-2,\ldots,-2N-1$, the first $N+1$ Hahn polynomials are orthogonal with respect to the Hahn measure
\begin{equation}\label{pesoH}
    \rho_{a,b,N}=\sum_{x=0}^N\frac{\Gamma(a+x+1)\Gamma(N-x+b+1)}{x!(N-x)!}\delta_x,
\end{equation}
and have non-null norms. The discrete measure $\rho_{a,b,N}$ is positive only when $a,b>-1$ or $a,b<-N$.

We also need the monic dual Hahn polynomials, for $n\geq0$,
\begin{equation}\label{dHP}
   R_n^{a,b,N}(x)=\sum_{j=0}^n\frac{(-n)_j(-N+j)_{n-j}(a+j+1)_{n-j}}{(-1)^j j!}\prod_{i=0}^{j-1}[x-i(i+a+b+1)].
\end{equation}
Notice that $R_n^{a,b,N}$ is always a polynomial of degree $n$.

Krall-Hahn measures can be constructed by using the following Christoffel transforms of the Hahn measure: for $a,b\not =-1,-2,\cdots $, $N$ a positive integer
and a quartet $\F=(F_1,F_2,F_3,F_4)$ of finite sets $F_i$, $i=1,2,3,4$, of positive integers, define the measure
\begin{equation}\label{hqs}
\rho _{a,b,N}^{\F}=\prod_{f\in F_1}(b+N+1+f-x)\prod_{f\in F_2}(x+a+1+f)\prod_{f\in F_3}(N-f-x)\prod_{f\in F_4}(x-f)\rho _{a,b,N},
\end{equation}
where $\rho _{a,b,N}$ is the Hahn measure (\ref{pesoH}).

It was conjectured in \cite{du0} and proved in \cite{du1} and \cite{DdI2} (see also \cite{DdI3}) that the orthogonal polynomials with respect to the measure $\rho _{a,b,N}^{\F}$ (when there exist) are also eigenfunctions of a higher order differential operator of the form (\ref{hodo}).
The result is also valid when $N$ is not a positive integer. The assumption of $N$ being a positive integer is only necessary for having
an explicit expression of the Hahn weight and include the most interesting case when the Krall-Hahn measure (\ref{hqs}) is positive. This is the reason why
we only consider the case $N$ being a positive integer in this paper.

The case $a$ or $b$ in $\{-1,-2,\cdots \}$ was neither considered in \cite{du0} nor in the subsequence papers because the Hahn measure is not well defined since the Gamma function has poles at the non-positive integers. Since the situation is similar to the Meixner case, we can also manage this case under similar assumptions on the finite sets $F_2$ and $F_4$, when $a\in \{-1,-2,\cdots \}$, or on the finite sets $F_1$ and $F_3$, when $b\in \{-1,-2,\cdots \}$, respectively.

In order to construct the sequence of orthogonal polynomials, we need some previous notations. Consider the polynomials
\begin{align}\label{thh}
\theta_n^{u}&=n(n+u+1),\\\nonumber
\psi_{r}^{u}(x)&=\prod_{h=1}^{r-1}\prod_{i=1}^h(2x+u-i-h),\\\label{phr}
\phi_{r,r_1}^{a}(x)&=\prod_{i=1}^{r_1-1}(x+a-r+1)_{r_1-i}.
\end{align}
Notice that for $u=a+b$ the polynomial (\ref{thh}) is the sequence of eigenvalues of the Hahn polynomials with respect to its associated second order difference operator. We also need the rational functions
\begin{equation}\label{xit}
\xi_{x,j}^h=\begin{cases} \displaystyle(-1)^j\frac{(x-j-N)_j(x-j+a+1)_j}{(x-j+a+b+N+2)_j},& \mbox{for $h=1$,}\\
\displaystyle\frac{(x-j+b+1)_j(x-j-N)_j}{(x-j+a+b+N+2)_j},& \mbox{for $h=2$,}\\
(x-j+a+1)_j,& \mbox{for $h=3$,}\\
\displaystyle(-1)^j(x-j+b+1)_j,& \mbox{for $h=4$,}
 \end{cases}
\end{equation}
which are related to the four $\D$-operators for the Hahn polynomials (see \cite[Section 7]{du1}, \cite[Section 5]{DdI2}; take into account the normalization of the Hahn polynomials we are considering in this paper).

Finally, we consider the polynomials $q_n^{a,b, N;\F}$  defined as
follows: write $G_i=I(F_i)$, $i=1,2,3,4$, where $I$ is the involution defined in (\ref{dinv}); write $m_i$ for the number of elements of $G_i$ and $m=\sum_{i=1}^4m_i$, and define
\begin{equation}\label{qusch}
q_n^{a,b,N;\F}(x)=\frac{\left|
  \begin{array}{@{}c@{}lccc@{}c@{}}
   &(-1)^{j-1}h_{n+1-j}^{a,b,N}(x-\max F_4-1) &&\hspace{-1.3cm}{}_{j=1,\ldots , m+1} \\
    \dosfilas{ \xi_{n-j+1,m-j+1}^1R_{g}^{-b,-a,a+b+N}(\theta_{-n+j-2}^{-a-b}) }{g\in G_1} \\
    \dosfilas{ \xi_{n-j+1,m-j+1}^2R_{g}^{-a,-b,a+b+N}(\theta_{-n+j-2}^{-a-b})}{g\in G_2}
    \\
    \dosfilas{ \xi_{n-j+1,m-j+1}^3R_{g}^{-b,-a,-2-N}(\theta_{-n+j-2}^{-a-b})  }{g\in G_3}\\
    \dosfilas{\xi_{n-j+1,m-j+1}^4R_{g}^{-a,-b,-2-N}(\theta_{-n+j-2}^{-a-b}) }{g\in G_4}
  \end{array}\right|}{\phi_{m,m_1+m_3}^{a}(n)\phi_{m,m_2+m_4}^{b}(n)\psi_{m}^{a+b}(n)}.
\end{equation}
Notice that the degree of $q_n^{a,b,N;\F}$ is $n$ if and only if $\Omega_\F^{a,b,N;\F}(n)\not =0$, where
\begin{equation}\label{omegh}
\Omega_\F^{a,b,N;\F}(n)=\frac{\left|
  \begin{array}{@{}c@{}lccc@{}c@{}}
   & &&\hspace{-1.3cm}{}_{j=1,\ldots , m} \\
    \dosfilas{ \xi_{n-j,m-j}^1R_{g}^{-b,-a,a+b+N}(\theta_{-n+j-1}^{-a-b})}{g\in G_1} \\
    \dosfilas{ \xi_{n-j,m-j}^2R_{g}^{-a,-b,a+b+N}(\theta_{-n+j-1}^{-a-b})}{g\in G_2}
    \\
    \dosfilas{ \xi_{n-j,m-j}^3R_{g}^{-b,-a,-2-N}(\theta_{-n+j-1}^{-a-b})  }{g\in G_3}\\
    \dosfilas{\xi_{n-j,m-j}^4R_{g}^{-a,-b,-2-N}(\theta_{-n+j-1}^{-a-b}) }{g\in G_4}
  \end{array}\right|}{\phi_{m,m_1+m_3}^{a}(n)\phi_{m,m_2+m_4}^{b}(n)\psi_{m}^{a+b}(n)}.
\end{equation}
When $a$ and $b$ are negative integers, the rational functions $\xi_{n-j+1,m-j+1}^h$ can vanish for certain $n$, $0\le n\le m$, with the consequence that $q_n^{a,b,N;\F}$=0. It is not difficult to see, that this problem can be avoided by normalizing the polynomials $q_n^{a,b,N;\F}$ 
and the rational function $\Omega_\F^{a,b,N;\F}$ as in (\ref{qusch}) and (\ref{omegh}), that is, dividing by  $\phi_{m,m_1+m_3}^{a}(n)\phi_{m,m_2+m_4}^{a}(n)\psi_{m}^{a+b}(n)$ (see \cite[p. 205]{DdI4}, where we used the same procedure to construct Jacobi-Sobolev orthogonal polynomials).

We also assume that $N$ is big enough, to avoid problems with the denominator $(n-m+a+b+N+2)_{m-j+1}$ of $\xi_{n-j+1,m-j+1}^h$, $h=1,2$ (it is enough to assume $N\ge -2-a-b$). As explained in \cite[Remark 6.5]{DdI2}, there are different sets $F_3$ and $F_4$ for which the measures $\rho _{a,b,N}^{\F}$
(\ref{hqs}) are equal. Each of these possibilities provides a different representation for the orthogonal polynomials with respect to $\rho _{a,b,N}^{\F}$ in the form (\ref{qusch}) and a different higher-order difference operator with respect to which they are eigenfunctions. It is not difficult to see that only one of these possibilities satisfies the condition $\max F_{3}, \max F_{4}<N/2$. This is the more interesting choice because it minimizes the order  of the associated higher-order difference operator. Hence, we also assume $\max F_{3}, \max F_{4}<N/2$.

We then have the following Theorems (the proofs are omitted because are analogous to that of Theorem \ref{1th}).
\begin{theorem}\label{h1th} Let $\hat a, \hat b$ be real numbers with $\hat a\in \{-1,-2,\cdots \}$ and $\hat b\not \in \{-1,-2,\cdots \}$, and let $\F$ be a quartet of finite sets $F_i$, $i=1,2,3,4$, of positive integers satisfying
\begin{equation}\label{hf2h}
\{0,1,2,\cdots , -\hat a-1\}\subset F_4\cup (-\hat a-1-F_2).
\end{equation}
Write $\Hh=\Hh(-\hat a-1,F_4,F_2)$ (see (\ref{spmh})) and
\begin{align*}
a=\hat a+\max F_2&+\max F_4+2,\quad b=\hat b+\max F_1+\max F_3+2,\\ & \tilde N=N-\max F_3-\max F_4-2.
\end{align*}
Consider the sequence of polynomials $(q_n^{a,b,\tilde N;\F})_n$ defined by (\ref{qusch}).
Assume that
\begin{equation}\label{hf2hi}
\Omega_\F^{a,b,\tilde N;\F}(n)\not =0,\quad n=0,1,2,\cdots , \tilde N+m_3+m_4+1.
\end{equation}
Then the polynomials $q_n^{a,b,\tilde N;\F}$, $n=0,1,\cdots ,\tilde N+m_3+m_4$, are orthogonal with respect to the measure
\begin{equation}\label{elr1}
\nu ^{\hat b,N}_{\hat a;\F}=\sum _{x=0;x\not \in F_4}^N \prod _{h\in \Hh}(x-h)\prod_{f\in F_1}(N-x+\hat b+1+f)\prod_{f\in F_3}(N-f-x)\frac{\Gamma(N-x+\hat b+1)}{(N-x)!}\delta_x.
\end{equation}
Moreover the polynomials $q_n^{a,b,\tilde N;\F}$, $n=0,1,\cdots ,\tilde N+m_3+m_4$, are eigenfunctions of a difference operator of the form (\ref{hodo}) with $-s=r=1+\sum_{i=1}^4\left(\sum_{f\in F_i}f-\binom{k_i}{2}\right)$.
\end{theorem}

The case when $\hat a\not \in \{-1,-2,\cdots \}$ and $\hat b
\in \{-1,-2,\cdots \}$ is analogous because of the symmetry of the Hahn measure with respect to $a$ and $b$.

When $\hat a, \hat b \in \{-1,-2,\cdots \}$, we have the following Theorem.

\begin{theorem}\label{h2th} Let $\hat a,\hat b\in \{-1,-2,\cdots \}$, and let $\F$ be a quartet of finite sets $F_i$, $i=1,2,3,4$, of positive integers satisfying
\begin{align*}\label{hf2h2}
\{0,1,2,\cdots , -\hat a-1\}&\subset F_4\cup (-\hat a-1-F_2),\\
\{0,1,2,\cdots , -\hat b-1\}&\subset F_3\cup (-\hat b-1-F_1).
\end{align*}
Write $\Hh_p=\Hh(-\hat a-1,F_4,F_2)$ and $\tilde \Hh_i=\Hh(-\hat b-1,F_3,F_1)$ (see (\ref{spmh})) and
\begin{align*}
a=\hat a+\max F_2&+\max F_4+2,\quad b=\hat b+\max F_1+\max F_3+2,\\ & \tilde N=N-\max F_3-\max F_4-2.
\end{align*}
Consider the sequence of polynomials $(q_n^{a,b,\tilde N;\F})_n$ defined by (\ref{qusch}).
Assume that
\begin{equation}\label{hf2hj}
\Omega_\F^{a,b,\tilde N;\F}(n)\not =0,\quad n=0,1,2,\cdots , \tilde N+m_3+m_4+1.
\end{equation}
Then the polynomials $q_n^{a,b,\tilde N;\F}$, $n=0,1,\cdots ,\tilde N+m_3+m_4$, are orthogonal with respect to the measure
\begin{equation}\label{elr}
\nu ^{N}_{\hat a,\hat b;\F}=\sum _{\substack{x=0\\x\not \in F_4\\x\not \in N-F_3}}^N \prod _{h\in \Hh_p}(x-h)\prod _{h\in \Hh_i}(N-x-h) \delta_x.
\end{equation}
Moreover the polynomials $q_n^{a,b,\tilde N;\F}$, $n=0,1,\cdots ,\tilde N+m_3+m_4$, are eigenfunctions of a difference operator of the form (\ref{hodo}) with $-s=r=1+\sum_{i=1}^4\left(\sum_{f\in F_i}f-\binom{k_i}{2}\right)$.
\end{theorem}

As for the Krall-Meixner polynomials (see Corollary \ref{eles}), it can be proved that the positivity of the measure $\nu ^{N}_{\hat a,\hat b;\F}$ in Theorems \ref{h1th} and \ref{h2th} implies the assumptions (\ref{hf2hi}) and (\ref{hf2hj}), respectively, and hence we can construct a sequence of orthogonal polynomials with respect to the positive measure $\nu ^{N}_{\hat a,\hat b;\F}$ using (\ref{qusch}).

For two nonnegative integers $c$ and $N$, a real number $d$, $d\not \in \{-1,-2,\cdots \}$ and a finite set $A$ of positive integers with $\max A\le N$, consider the measure $\rho_{c,d,N}^A$ obtained from the Hahn measure $\rho_{c,d,N}$ (\ref{pesoH}) by removing the mass at the points belonging to $A$. Proceeding as in the Remark \ref{rm4}, it is easy to show that the measure $\rho_{c,d,N}^A$ can be represented by a measure of the form (\ref{elr1}), where $\hat b=d$, $F_1=F_3=\emptyset$, $F_4=A$, $\hat a=-\max A-1$ (and then $\hat a\in\{-2,-3,\cdots \}$) and certain finite set $F_2$ of positive integers satisfying (\ref{hf2h}). Hence we can construct a sequence of orthogonal polynomials with respect to the positive measure  $\rho_{c,d,N}^A$ using (\ref{qusch}), and moreover these orthogonal polynomials are eigenfunctions of a higher order difference operator of the form (\ref{hodo}).

In the same way, assuming that $c,d,N$ are nonnegative integers and $A$ and $B$ two sets of positive integers with $\max A,\max B<N/2$, the measures $\rho_{c,d,N}^{A,B}$ obtained from the Hahn measure $\rho_{c,d,N}$ (\ref{pesoH}) by removing the mass at the points belonging to $A$ and $N-B$, can be represented by a measure of the form (\ref{elr}), so that Theorem \ref{h2th} can be applied to $\rho_{c,d,N}^{A,B}$.

\bigskip

\noindent
\textit{Mathematics Subject Classification: 42C05, 33C45, 33E30}

\noindent
\textit{Key words and phrases}: Orthogonal polynomials. Krall discrete polynomials Exceptional orthogonal polynomial. 
Meixner polynomials.  Laguerre polynomials. Hahn polynomials.

     \end{document}